\documentclass[12pt]{amsart}

\usepackage[top=1.5in,right=1.25in,bottom=1.5in,left=1.25in]{geometry}

\usepackage{amsmath, amssymb, amsthm, mathrsfs}
\usepackage[svgnames, dvipsnames, usenames]{xcolor}

\usepackage[colorlinks]{hyperref}
\PassOptionsToPackage{colorlinks}{hyperref}
\hypersetup{urlcolor = RedViolet, linkcolor = RoyalBlue, citecolor = ForestGreen}
\usepackage[capitalize]{cleveref}

\newcommand{\bB}{\mathcal{B}}

\DeclareMathOperator{\PSL}{PSL}

\DeclareMathOperator{\Sym}{Sym}
\DeclareMathOperator{\Aut}{Aut}

\DeclareMathOperator{\Der}{Der}
\DeclareMathOperator{\fix}{fix}

\newcommand{\nv}{^{-1}}

\renewcommand\epsilon\varepsilon

\renewcommand\emptyset\varnothing

\usepackage{enumerate}
\usepackage{mathtools}
\usepackage{longtable}
\usepackage{blkarray}
\usepackage{centernot}

\newtheorem{thm}{Theorem}[section]

\newtheorem{lem}[thm]{Lemma}
\newtheorem{prop}[thm]{Proposition}
\newtheorem{cor}[thm]{Corollary}

\theoremstyle{definition}
\newtheorem{defn}[thm]{Definition}
\newtheorem{rmk}[thm]{Remark}

\date{\today}

\begin{document}

\title{Uniformly vertex-transitive graphs}
\author{Simon Schmidt, Chase Vogeli, Moritz Weber}

\address{Saarland University, Fachbereich Mathematik, Postfach 151150,
66041 Saarbr\"ucken, Germany}
\email{simon.schmidt@math.uni-sb.de}
\email{weber@math.uni-sb.de}

\address{Department of Mathematics, Massachusetts Institute of Technology, Cambridge, MA 02139, USA}
\email{cpvogeli@mit.edu}

\date{\today}

\begin{abstract}
We introduce uniformly vertex-transitive graphs as vertex-transitive graphs satisfying a stronger condition on their automorphism groups, motivated by a problem which arises from a Sinkhorn-type algorithm. We use the derangement graph $D(\Gamma)$ of a given graph $\Gamma$ to show that the uniform vertex-transitivity of $\Gamma$ is equivalent to the existence of cliques of sufficient size in $D(\Gamma)$. Using this method, we find examples of graphs that are vertex-transitive but not uniformly vertex-transitive, settling a previously open question. Furthermore, we develop sufficient criteria for uniform vertex-transitivity in the situation of a graph with an imprimitive automorphism group. We classify the non-Cayley uniformly vertex-transitive graphs on less than 30 vertices outside of two complementary pairs of graphs.
\end{abstract}

\maketitle

\section{Introduction and main results}\label{sec:int}

The motivation for our article arose from \cite{NSW19}, where a Sinkhorn-type algorithm was presented for determining whether or not a given finite graph has quantum symmetries. This algorithm can only be applied to graphs which are vertex-transitive in a stronger sense we want to specify in the present article. We introduce it independently of the framework of \cite{NSW19}. 

Specifically, a graph $\Gamma$ on $n$ vertices is uniformly vertex-transitive if there exist $n$ graph automorphisms $\sigma_1,\ldots,\sigma_n\in\Aut(\Gamma)$ such that, when viewed as matrices, we have $\sum\sigma_i=J_n$, the $n\times n$ matrix with all entries 1. In \cref{prop:implications} we show that we have the following inclusions of classes of graphs:
\[ \small\{\text{Cayley graphs}\} \subset \{\text{uniformly vertex-transitive graphs}\} \subset \{\text{vertex-transitive graphs}\} \] 
We show that both of these inclusions are strict, with the Petersen graph and its line graph as testimonials. For the former, we explicitly prove that the Petersen graph is uniformly vertex-transitive in \cref{prop:petersen}, by constructing a set of automorphisms in the above sense, while it is well-known that the Petersen graph is non-Cayley (see \cite{MP94} for instance). For the latter, it is well-known that the Petersen graph is edge-transitive (see \cite{HS93}, for instance) which implies that its line graph is vertex-transitive, but showing that its line graph is not uniformly vertex-transitive is more subtle. 

We do so by developing some criteria for a graph to be uniformly vertex-transitive. In \cref{sec:comp}, we see that the derangement graph $D(\Gamma)$ of a graph $\Gamma$ encodes information about the Schur product of automorphisms of $\Gamma$. In the main theorem of this article (\cref{thm:clique_nhd}), we prove that $\Gamma$ is uniformly vertex-transitive if $D(\Gamma)$ has a maximal clique number.

Additionally, we develop sufficient criteria for a graph to be uniformly-vertex transitive in the case of its automorphism group being imprimitive in \cref{sec:imp}. In particular, we show in \cref{thm:factor} that the existence of an $\Aut(\Gamma)$-invariant partition of the vertex set satisfying some extra conditions implies $\Gamma$ is uniform vertex-transitive.

\begin{table}[ht]\begin{tabular}{|c|cc|} \hline
Vertices & UVT & non-UVT \\ \hline\hline
10 & 2 & \\
15 & & 4 \\
16 & 8 & \\
18 & 4 & \\
20 & 70 & 12 \\
24 & 112 & \\
26 & 132 & \\
28 & $\geq 24$ & $\geq 38$ \\
30 & $\geq 324$ & $\geq 730$ \\\hline
\end{tabular}
\\[3mm]\caption{The counts of uniformly vertex-transitive (UVT) and non-uniformly vertex-transitive (non-UVT) graphs on numbers of vertices for which there exist vertex-transitive vertex-transitive non-Cayley graphs.}\label{tab:simple_counts}\end{table}

We additionally present some experimental results, based on a search for uniformly vertex-transitive graphs in a database \cite{R97} of vertex-transitive non-Cayley graphs on up to 30 vertices. We present the counts of graphs in the database that were identified to be uniformly vertex-transitive or non-uniformly vertex-transitive in \cref{tab:simple_counts}. In particular, we fully classify the non-Cayley uniformly vertex-transitive graphs on less than 30 vertices outside of two complementary pairs of graphs on 28 vertices with large automorphism groups. More detailed information about experimental results can be found in \cref{sec:exp}.

In \cref{sec:further}, we introduce the property of $k$-uniform vertex-transitivity, one possible generalization of uniform vertex-transitivity, and record some basic facts about this property.

\section{The notion of uniform vertex-transitivity}\label{sec:uvt}

In this section we definine the notion of uniformly vertex-transitive graphs. Before doing so, let us recall the notions of Cayley graphs and vertex-transitive graphs. Throughout this article, we restrict to finite simple graphs. For a graph $\Gamma$, we denote its vertex set and edge set as $V(\Gamma)$ and $E(\Gamma)$, respectively.

\begin{defn}
Let $G$ be a finite group and $S\subset G$ be a subset of $G$. The (uncolored, undirected) \emph{Cayley graph} $C(G,S)$ is a graph with vertex set $G$ in which two group elements $a,b\in G$ are adjacent if there exists $s\in S$ such that either $a=sb$ or $b=sa$. We also say that a finite graph $\Gamma$ is \emph{Cayley} if $\Gamma$ is isomorphic to $C(G,S)$ for a group $G$ and generating set $S$.
\end{defn}

Note that it is also common to define a Cayley graph as directed to indicate which equality holds and colored to indicate the element $s\in S$. For our purposes however, we will work with this uncolored, undirected version of the definition. We have the following well-known characterization of Cayley graphs by Sabidussi. We give a proof for the convenience of the reader. Recall that a transitive group action is \emph{regular} if all of the point stabilizers are trivial.

\begin{prop}[{\cite[Lemma 4]{S58}}]
\label{prop:cay-reg}
A graph $\Gamma$ is a Cayley graph if and only if a subgroup of $\Aut(\Gamma)$ acts regularly on $V(\Gamma)$.
\end{prop}
\begin{proof}
Suppose that $\Gamma$ is Cayley, so $\Gamma$ is isomorphic to $C(G,S)$ for some group $G$ and subset $S\subset G$. Observe that for adjacent vertices $a,b\in G$ and $s\in S$ such that $b=sa$, for any $g\in G$, we have that $gb=gsa$ as well. Hence each $g\in G$ induces an automorphism of $\Gamma$ via left multiplication. Thus, $G$ is a subgroup of $\Aut(\Gamma)$. Because the vertices of $\Gamma$ are identified with elements of $G$, this action is equivalent to the left regular action of $G$ on itself. Thus, $G$ is a regular subgroup of $\Aut(\Gamma)$.

In the opposite direction, suppose that $\Aut(\Gamma)$ has a subgroup $G$ with a regular action on $V(\Gamma)$. Thus, we seek to construct a graph isomorphism $\varphi:C(G,S)\to\Gamma$ for some subset $S\subset G$. Pick an arbitrary vertex $v_0\in V(\Gamma)$, which we will `label' as the identity in $G$. For all vertices $v \in V(\Gamma)$, there exists a unique $g_v\in G$ which sends $v_0\mapsto v$ by the regularity of the action of $G$. The regularity of $G$ also gives that $|G|=|V(\Gamma)|$. Let $S=\{g_v : v \in V(\Gamma), v \sim v_0\}$, Defining $\varphi$ as $\varphi(g_v) = g_v(v)$ indeed gives a graph isomorphism $C(G,S)\simeq\Gamma$, so $\Gamma$ is Cayley.
\end{proof}

It is clear from the above proposition that the automorphism group of a Cayley graph is transitive. Indeed, every Cayley graph also has the weaker property of being vertex-transitive, which we define below. 

\begin{defn}
A graph $\Gamma$ is \emph{vertex-transitive} if for any vertices $u,v\in V(\Gamma)$, there exists an automorphism $\sigma\in\Aut(\Gamma)$ such that $\sigma(u)=v$. In other words, $\Aut(\Gamma)$ acts transitively on the vertices of $\Gamma$.
\end{defn}

Having seen the definition of a vertex-transitive graph, we are ready to define what it means for a graph to be uniformly vertex-transitive.

\begin{defn}
Let $\Gamma$ be a graph on $n$ vertices. The graph $\Gamma$ is \emph{uniformly vertex-transitive} if there exists a size $n$ subset $\{\sigma_1,\ldots,\sigma_n\} \subset \Aut(\Gamma)$ such that, when viewed as matrices, $\sum \sigma_i = J_n$, where $J_n$ is the $n\times n$ matrix with all entries 1. Such a subset $\{\sigma_1,\ldots,\sigma_n\}$ is called a \emph{maximal Schur set}\footnote{Maximal Schur sets are also known elsewhere in the literature as \emph{sharply transitive sets}.}.
\end{defn}

We motivate the naming of this property with the following observation. If $\Gamma$ is vertex-transitive and $j\in V(\Gamma)$, we find $n$ automorphisms $\{ \sigma_1^j,\ldots, \sigma_n^j \}$ for which $\sum_i \sigma_i^j$ has all entries 1 in row $j$. Then, $\Gamma$ is uniformly vertex-transitive if it is possible to make a \emph{uniform} choice of $\{\sigma_1,\ldots,\sigma_n\} \subset \Aut(\Gamma)$ such that $\sum \sigma_i$ has all entries 1 in \emph{each} row. A maximal Schur set is named as such to reflect the fact that the matrix $J_n$ is the identity under the Schur (i.e. entry-wise) product of $n\times n$ matrices. We make the relationship between these properties of such sets of automorphisms rigorous with the following result.

\begin{lem} \label{lem:eq-conditions} 
Let $\Gamma$ be a finite graph on $n$ vertices and let $S = \{\sigma_1,\ldots,\sigma_n\}\subset\Aut(\Gamma)$. The following are equivalent:
\begin{enumerate}[(a)]
\item $\sum \sigma_i = J_n$.
\item For all $u,v\in V(\Gamma)$, there exists an $i$ for which $\sigma_i(u)=v$.
\item For all $u,v\in V(\Gamma)$, there exists a unique $i$ for which $\sigma_i(u)=v$.
\end{enumerate}
\end{lem}
\begin{proof}
Suppose that $\sum \sigma_i = J_n$, and fix $u,v\in V(\Gamma)$. The entry of $J_n$ corresponding to $(u,v)$ will clearly be 1, so there is a term $\sigma_i$ in the sum $\sum\sigma_i=J_n$ for which the $(u,v)$ entry is 1, thus $\sigma_i(u)=v$. Since all the $\sigma_i$ are permutation matrices with all entries 0 or 1, this $\sigma_i$ is unique. This establishes (a)$\implies$(c).

The implication (c)$\implies$(b) is clear.

To show (b)$\implies$(a), fix a vertex $u\in V(\Gamma)$. Since for every $v\in V(\Gamma)$, there is an $i$ for which $\sigma_i(u)=v$ and there are as many vertices in $V(\Gamma)$ as there are elements in $S$, the row corresponding to $u$ in the sum $\sum\sigma_i$ must have all 1s. Therefore $\sum\sigma_i=J_n$.
\end{proof}

As mentioned above, every Cayley graph is vertex-transitive. We observe that the property of being uniformly vertex-transitive sits between these two other graph properties. 

\begin{prop} \label{prop:implications}
Let $\Gamma$ be a finite graph on $n$ vertices. \begin{enumerate}[(a)]
\item If $\Gamma$ is a Cayley graph, then $\Gamma$ is uniformly vertex-transitive.
\item If $\Gamma$ is a uniformly vertex-transitive graph, then $\Gamma$ is vertex-transitive.
\end{enumerate}
\end{prop}
\begin{proof}
Let $\Gamma$ be a Cayley graph. From \cref{prop:cay-reg}, we have that $\Aut(\Gamma)$ contains a regular subgroup which we will denote by $G$. By the definition of a regular action, $G$ satisfies (c) of \cref{lem:eq-conditions}, so $G$ is a maximal Schur set. Thus $\Gamma$ is uniformly vertex-transitive, proving (a).

Suppose $\Gamma$ is uniformly vertex-transitive, so it follows there exists a maximal Schur set $S\subset \Aut(\Gamma)$. Thus by \cref{lem:eq-conditions}, $S$ also satisfies property (b) of the lemma, so $\Gamma$ is vertex-transitive. This proves part (b).
\end{proof}

Hence we have the following chain of implications of graph properties:
\[ \text{Cayley} \Longrightarrow \text{uniformly vertex-transitive} \Longrightarrow \text{vertex-transitive} \]
It is a natural question to ask whether the converse of any of these implications holds. In \cite{NSW19}, the authors were particularly interested in the latter implication. Comparing Cayley graphs and vertex-transitive graphs only, it is well known that the Petersen graph is the smallest vertex-transitive graph which is not Cayley. Such graphs have been extensively studied (see \cite{MP94}, for instance). In fact, the Petersen graph is even uniformly vertex-transitive, which shows that the first of the above implications is not an equivalence.

\begin{prop} \label{prop:petersen}
The Petersen graph is a uniformly vertex-transitive graph that is not a Cayley graph.
\end{prop}
\begin{proof}
Recall that the Petersen graph $P$ is realized with $V(P)$ consisting of the two-element subsets of $\{1,2,3,4,5\}$ in which two subsets are adjacent if they are disjoint. It is well known that $\Aut(P)\simeq S_5$ and that the action of $\Aut(P)$ on $V(P)$ is induced by the standard action of $S_5$ on $\{1,2,3,4,5\}$. Let $\alpha, \beta \in S_5$ be given by $\alpha = (1\,5\,3\,4)$ and $\beta = (1\,2\,3\,4\,5)$. When viewed as permutations of the 10 vertices of the Petersen graph, $\alpha$ and $\beta$ are given by the permutation matrices
\[ \Tiny \alpha = \begin{blockarray}{ccccccccccc}
& \scriptstyle{12}
& \scriptstyle{23}
& \scriptstyle{34}
& \scriptstyle{45}
& \scriptstyle{15}
& \scriptstyle{13}
& \scriptstyle{24}
& \scriptstyle{35}
& \scriptstyle{14}
& \scriptstyle{25} \\
\begin{block}{c[cccccccccc]}
\scriptstyle{12}&&&&&&&1&&& \\
\scriptstyle{23}&&&&&&&&&&1 \\
\scriptstyle{34}&&&&&&&&1&& \\
\scriptstyle{45}&&&&&&1&&&& \\
\scriptstyle{15}&&&&&&&&&1& \\
\scriptstyle{13}&&&&1&&&&&& \\
\scriptstyle{24}&&1&&&&&&&& \\
\scriptstyle{35}&&&&&1&&&&& \\
\scriptstyle{14}&&&1&&&&&&& \\
\scriptstyle{25}&1&&&&&&&&& \\
\end{block}
\end{blockarray} \qquad \beta = \begin{blockarray}{ccccccccccc}
& \scriptstyle{12}
& \scriptstyle{23}
& \scriptstyle{34}
& \scriptstyle{45}
& \scriptstyle{15}
& \scriptstyle{13}
& \scriptstyle{24}
& \scriptstyle{35}
& \scriptstyle{14}
& \scriptstyle{25} \\
\begin{block}{c[cccccccccc]}
\scriptstyle{12}&&&&&1&&&&& \\
\scriptstyle{23}&1&&&&&&&&& \\
\scriptstyle{34}&&1&&&&&&&& \\
\scriptstyle{45}&&&1&&&&&&& \\
\scriptstyle{15}&&&&1&&&&&& \\
\scriptstyle{13}&&&&&&&&&&1 \\
\scriptstyle{24}&&&&&&1&&&& \\
\scriptstyle{35}&&&&&&&1&&& \\
\scriptstyle{14}&&&&&&&&1&& \\
\scriptstyle{25}&&&&&&&&&1& \\
\end{block}
\end{blockarray}
\]
in which we write 12 for the vertex $\{1,2\}$ of the Petersen graph, for instance. From these matrix descriptions, it is clear that
\begin{equation*}
\sum_{i=0}^4 \beta^i = \begin{bmatrix} J_5 & \\ & J_5 \end{bmatrix} \qquad \sum_{i=0}^4 \alpha\beta^i = \begin{bmatrix} & J_5 \\ J_5 & \end{bmatrix}
\end{equation*}
where $J_k$ is the $k\times k$ matrix with all entries 1. Therefore, we have 
\begin{equation*}
\sum_{i=0}^1\sum_{j=0}^4 \alpha^i\beta^j = J_{10}.
\end{equation*}
which shows that $\{\alpha^i \beta^j : 0 \leq i \leq 1 \text{ and } 0 \leq j \leq 4\} \subset \Aut(\Gamma)$ is a maximal Schur set. It follows that $P$ is uniformly vertex-transitive.

That $P$ is not a Cayley graph is well known, see for instance \cite{MP94}.
\end{proof}

The remaining question is whether there exist vertex-transitive graphs which are not uniformly vertex-transitive. We give an affirmative answer to this question in the following section, as well as techniques used to show this result and to determine whether or not general graphs are uniformly vertex-transitive.

\begin{rmk}
If $\Gamma$ is a vertex-transitive graph for which $|V(\Gamma)|=|\!\Aut(\Gamma)|$, \cref{prop:cay-reg} implies that $\Gamma$ is a Cayley graph for $\Aut(\Gamma)$. In this case, $\Gamma$ is furthermore uniformly vertex-transitive.
\end{rmk}

\section{Computational methods for detecting uniform vertex-transitivity}\label{sec:comp}

We consider an arbitrary finite graph $\Gamma$ on $n$ vertices. Our present motivation is to develop a computational strategy for determining whether $\Gamma$ is uniformly-vertex transitive. This problem is apparently harder than determining whether or not it is Cayley. Indeed, unlike the latter---which in light of \cref{prop:cay-reg} can be easily realized from knowing the subgroups of $\Aut(\Gamma)$---the former depends on the existence of maximal Schur sets, which are not subgroups of $\Aut(\Gamma)$ in general. A naive approach of checking all subsets of $\Aut(\Gamma)$ of a given size quickly becomes untenable for $\Gamma$ with sufficiently many vertices or automorphisms, so a more informed approach is necessary.

In this section, we consider the derangement graph $D(\Gamma)$ of $\Gamma$, which we will see encodes how elements of $\Aut(\Gamma)$ relate under the Schur product of permutation matrices. Culminating in  \cref{thm:clique_nhd}, we will see that the uniform vertex-transitivity of $\Gamma$ is equivalent to the existence of cliques of a certain size in $D(\Gamma)$. Because the problem of finding cliques in graphs is relatively well-studied and implemented, we propose this as an effective method of determining whether an arbitrary graph $\Gamma$ is uniformly-vertex transitive.

For a graph $\Gamma$ We denote by $\Der_{\Gamma}$ the subset of $\Aut(\Gamma)$ which consists of all permutations without fixed points. Such permutations are often called \emph{derangements}. 

\begin{defn}
Let $\Gamma$ be a graph. The \emph{derangement graph} $D(\Gamma)$ is a graph with vertex set $\Aut(\Gamma)$ in which automorphisms $\sigma$ and $\tau$ are adjacent if $\sigma\nv\tau \in\Der_\Gamma$. Thus, $D(\Gamma)$ coincides with the Cayley graph $C(\Aut(\Gamma),\Der_\Gamma)$.
\end{defn}

We observe that the derangement graph encodes the orthogonality of automorphisms with respect to the Schur product.

\begin{lem} \label{lem:der-schur}
For automorphisms $\sigma, \tau \in\Aut(\Gamma)$, we have $\sigma\nv\tau\in\Der_\Gamma$ if and only if the Schur product of $\sigma$ and $\tau$ is zero.
\end{lem}
\begin{proof}
Suppose $\sigma\nv\tau$ is not a derangement, so there is a fixed point $i\in V(\Gamma)$ of $\sigma\nv\tau$, for which $\sigma\nv\tau(i)=i$. For such a vertex $i$, we have $\sigma(i) = \tau(i)$, which means the Schur product of $\sigma$ and $\tau$ is nonzero.
\end{proof}

In other words, the edges of $D(\Gamma)$ connect automorphisms with Schur product zero. Recall that a \emph{clique} in a graph $\Gamma$ is a subset $S\subset V(\Gamma)$ which induces a complete subgraph, i.e. any two vertices in $S$ are adjacent. 

\begin{lem} \label{lem:clique}
Let $\Gamma$ be a graph on $n$ vertices. A size $n$ subset $S \subset \Aut(\Gamma)$ is a maximal Schur set if and only if $S$ is a clique in $D(\Gamma)$.
\end{lem}
\begin{proof}
Suppose that $S=\{\sigma_1,\ldots,\sigma_n\} \subset V(D(\Gamma))$ is a clique. That is, $\sigma_i$ and $\sigma_j$ have Schur product 0 for $i\neq j$. That is, for each row $k$, the matrices $\sigma_i$ and $\sigma_j$ have a 1 in different positions, so row $k$ of $\sigma_i+\sigma_j$ will consist of 1s in two entries and 0s otherwise. Continuing inductively, we see that each row of $\sum\sigma_i$ consists of all entries 1, so it follows that $\sum\sigma_i=J_n$, and thus $S$ is a maximal Schur set.

Conversely suppose that $S=\{\sigma_1,\ldots,\sigma_n\}\subset \Aut(\Gamma)$ is a maximal Schur set, so $\sum\sigma_i=J_n$. Take $\sigma_i$ and $\sigma_j$ for $i\neq j$. If $\sigma_i$ and $\sigma_j$ have nonzero Schur product, then $\sigma_i$ and $\sigma_j$ have a 1 in the same entry. However, if this was the case, $\sum\sigma_i$ would have an entry other than 1, which is a contradiction, as we assumed $\sum\sigma_i=J_n$. Thus $S$ is a clique in $D(\Gamma)$, as all $\sigma_i$ are adjacent.
\end{proof}

\begin{lem} \label{lem:clique_mult}
Suppose $S = \{\sigma_1,\ldots,\sigma_n\} \subset \Aut(\Gamma)$ is a maximal Schur set. For any $\alpha\in \Aut(\Gamma)$, the set $\alpha S=\{\alpha\sigma_1,\ldots,\alpha\sigma_n\}$ is also a maximal Schur set.
\end{lem}
\begin{proof}
View $g,\sigma_1,\ldots,\sigma_n\in G$ all as their corresponding permutation matrices. That $S$ is a maximal Schur set gives that $\sum\sigma_i=J_n$. Observe that
\begin{equation*}
\sum_{i=1}^n \alpha\sigma_i = \alpha \sum_{i=1}^n \sigma_i = \alpha J_n = J_n
\end{equation*}
which shows that $\alpha S$ is also a maximal Schur set.
\end{proof}

\begin{cor} \label{lem:id_clique}
If $\Gamma$ is uniformly vertex-transitive, then $\Aut(\Gamma)$ contains a maximal Schur set which contains the identity automorphism.
\end{cor}
\begin{proof}
Let $S=\{\sigma_1,\ldots,\sigma_n\}$ be a maximal Schur set in $\Aut(\Gamma)$. Observe that the set $\sigma_1^{-1}S = \{\text{id}, \sigma_1^{-1}\sigma_2, \ldots, \sigma_1^{-1}\sigma_n\}$ is a maximal Schur set, by  \cref{lem:clique_mult}, which contains the identity.
\end{proof}

Before preceding, we recall some more terminology from graph theory. The \emph{clique number} $\omega(\Gamma)$ of a graph $\Gamma$ is the maximum size of a clique in $\Gamma$. The \emph{neighborhood} of a vertex $v\in V(\Gamma)$ is the induced subgraph on the set of vertices adjacent to $v$ in $\Gamma$, and is denoted by $\Gamma_v$.

\begin{thm} \label{thm:clique_nhd}
Let $\Gamma$ be a graph on $n$ vertices. $\Gamma$ is uniformly vertex-transitive if and only if $\omega(D(\Gamma)_{\text{id}}) = n-1$.
\end{thm}
\begin{proof}
Suppose that $\Gamma$ is uniformly vertex-transitive. In light of the previous corollary, let $S\subset \Aut(\Gamma)$ be a maximal Schur set which contains $\text{id}\in\Aut(\Gamma)$. By \cref{lem:clique_mult}, $S$ is a clique in the graph $D(\Gamma)$. It follows that $S-\{\text{id}\}$ is a clique in $D(\Gamma)_{\text{id}}$ of size $n-1$. Thus, $\omega(D(\Gamma)_{\text{id}}) = n-1$.

In the opposite direction, suppose that $\omega(D(\Gamma)_{\text{id}}) = n-1$, so there exists a size $n-1$ clique $C\subset V(D(\Gamma))$ consisting of neighbors of $\text{id}\in D(\Gamma)$. The set $C\cup \{\text{id}\}$ is then a clique of size $n$ in $D(\Gamma)$. By \cref{lem:clique}, $C\cup \{\text{id}\}$ is a maximal Schur set in $\Aut(\Gamma)$, so $\Gamma$ is thus uniformly vertex-transitive.
\end{proof}

We now apply the above methods to the line graph of the Petersen graph. Recall that for a finite graph $\Gamma$, its \emph{line graph} $L(\Gamma)$ is a graph with vertices $E(\Gamma)$ in which edges $e_1,e_2\in E(\Gamma)$ are adjacent in $L(\Gamma)$ if $e_1$ and $e_2$ share a common vertex in $\Gamma$.

\begin{cor} \label{thm:line_petersen}
The line graph of the Petersen graph is vertex-transitive, but not uniformly vertex-transitive.
\end{cor}
\begin{proof}
Let $P$ be the Petersen graph, and let $L(P)$ denote its line graph. Using Sage \cite{sage}, we determined that $\omega(D(L(P))_{\textrm{id}}) = 12$. Now, $L(P)$ has 15 vertices, thus by \cref{thm:clique_nhd}, we have that $L(P)$ is not uniformly vertex-transitive. That $L(P)$ is vertex-transitive follows from the fact that $P$ is edge-transitive, as can be seen in \cite[Theorem 4.7]{HS93}.
\end{proof}

In light of \cref{thm:line_petersen}, we may conclude that the converses of the aforementioned chain of implications of graph properties do not hold, that is:
\[ \text{Cayley} \centernot\Longleftarrow \text{uniformly vertex-transitive}
\centernot\Longleftarrow \text{vertex-transitive} \]
Thus, the property of uniform vertex-transitive sits strictly between vertex-transitivity and being a Cayley graph, settling our motivating question from \cite{NSW19}.

\section{Uniform vertex-transitivity for imprimitive graphs}\label{sec:imp}

In this section, we study uniform vertex-transitivity for a special class of graphs, those graphs whose automorphism group is imprimitive. As in the previous section, we consider a finite graph $\Gamma$ on $n$ vertices. We begin by recalling the notion of primitivity from permutation group theory.

\begin{defn}
Let $G$ be a group which acts transitively on a set $X$. A \emph{block} of $G$ is a subset $B\subset X$ for which $gB=B$ or $gB\cap B=\emptyset$ for all $g\in G$. A block $B$ is \emph{trivial} if $B$ is a singleton or if $B=X$, and is \emph{nontrivial} otherwise. If $G$ has a nontrivial block, then $G$ acts \emph{imprimitively}, and \emph{primitively} otherwise.
\end{defn}

As mentioned earlier, we are interested in the case in which a graph has an imprimitive automorphism group. We call a graph $\Gamma$ an \emph{imprimitive graph} if the action of $\Aut(\Gamma)$ on $V(\Gamma)$ is imprimitive. 

\begin{defn}
For $G$ a group acting transitively on a set $X$ with a block $B$ of $G$, the set $\bB=\{ gB : g\in G\}$ is called a \emph{block system} of $B$.
\end{defn}

It is well known that a block system $\bB$ forms a partition of $X$ and that each block has the same cardinality. Consider an imprimitive group $G$ with nontrivial block system $\bB$ consisting of $m$ blocks of size $k$. Then, the action of $G$ on $\bB$ induces a group homomorphism $G \to \Sym(\bB) \simeq S_m$. The kernel of this homomorphism is the \emph{fixer} $\fix_G(\bB)$ of $\bB$ in $G$, which is the subgroup of $G$ consisting of automorphisms which leave each block in place. In the situation of an automorphism group of a graph $\Gamma$, we will often write $\fix_{\Gamma}(\bB)$ to mean $\fix_{\Aut(\Gamma)}(\bB)$. The quotient $G/\!\fix_G(\bB)$ then has a faithful and transitive action on the blocks $\bB$. 

We wish to study uniform vertex-transitivity in the context of such imprimitive graphs. First, we abstract the notion of uniform vertex-transitivity to general permutation groups.

\begin{defn}
Let $G$ be a group acting on a set $X$ with $|X|=n$. The group $G$ is \emph{uniformly transitive} if there exists a size $n$ subset $\{\sigma_1,\ldots,\sigma_n\} \subset G$ such that, when viewed as matrices, $\sum \sigma_i = J_n$. Such a subset $\{\sigma_1,\ldots,\sigma_n\}$ is called a \emph{maximal Schur set}.
\end{defn}

\begin{defn}
Let $\Gamma$ be an imprimitive graph with nontrivial block system $\bB$ of $\Aut(\Gamma)$, which has $m$ blocks of size $k$. The block system $\bB$ is \emph{factorizing} if 
\begin{enumerate}[(i)]
\item the group $\fix_\Gamma(\bB)$ contains a size $k$ subset of element which are mutually orthogonal with respect to the Schur product of matrices, and
\item the group $\Aut(\Gamma)/\!\fix_\Gamma(\bB)$ acts uniformly transitively on $\bB$.
\end{enumerate} 
\end{defn}

Intuitively, a block system $\bB$ for $\Aut(\Gamma)$ is factorizing both the action on the blocks and the action within the blocks admit maximal Schur sets. We now turn to the main result of this section.

\begin{thm}\label{thm:factor}
Let $\Gamma$ be an imprimitive graph on $n$ vertices. If $\Aut(\Gamma)$ has a factorizing block system, then $\Gamma$ is uniformly vertex-transitive.
\end{thm}
\begin{proof}
Suppose $\bB=\{B_1,\ldots,B_m\}$ consists of $m$ blocks of size $k$. Assume without loss of generality that the vertices of $\Gamma$ are ordered in such a way that the first $k$ vertices are in $B_1$, the next $k$ in $B_2$ and so forth. Let $A=\{\alpha_1,\ldots,\alpha_k\}$ be a mutually orthogonal subset of size $k$ in $\fix_\Gamma(\bB)$ and let $B'=\{\beta_1',\ldots,\beta_m'\}$ be a maximal Schur set of the action of $G/\!\fix_\Gamma(\bB)$ on $\bB$. Let $B=\{\beta_1,\ldots,\beta_m\}$ be a lift of $B'$ to $\Aut(\Gamma)$, i.e. the image of $B\subset\Aut(\Gamma)$ under the canonical surjection $G\to G/\!\fix_\Gamma(\bB)$ is $B'$. 

We have that $\sum\alpha_i$ has the block diagonal matrix form
\[ \sum_{i=1}^k \alpha_i = \begin{bmatrix} J_k & \cdots & 0 \\ \vdots & \ddots & \vdots \\ 0 & \cdots & J_k \end{bmatrix} \] 
consisting of a matrix of $m\times m$ blocks of size $k\times k$. Furthermore, $\sum\beta_j$ has the block form
\[ \sum_{j=1}^m \beta_j = \begin{bmatrix} P_{1,1} & \cdots &  P_{1,m} \\ \vdots & \ddots & \vdots \\ P_{m,1} & \cdots & P_{m,m} \end{bmatrix} \] 
in which each $P_{i,j}$ is some $k\times k$ permutation matrix. Finally, we have that
\begin{align*} \sum_{i=1}^k \sum_{j=1}^m \alpha_i\beta_j &= \left( \sum_{i=1}^k \alpha_i \right) \left( \sum_{j=1}^m \beta_j \right) = \begin{bmatrix} J_k & \cdots & 0 \\ \vdots & \ddots & \vdots \\ 0 & \cdots & J_k \end{bmatrix} \begin{bmatrix} P_{1,1} & \cdots &  P_{1,m} \\ \vdots & \ddots & \vdots \\ P_{m,1} & \cdots & P_{m,m} \end{bmatrix} \\&= \begin{bmatrix} J_k & \cdots & J_k \\ \vdots & \ddots & \vdots \\ J_k & \cdots & J_k \end{bmatrix} = J_{mk} = J_n
\end{align*}
which shows that the set $AB=\{ \alpha_i\beta_j : 1 \leq i \leq k, 1 \leq j \leq m\}$ is a maximal Schur set of $\Aut(\Gamma)$. Thus $\Gamma$ is uniformly vertex-transitive.
\end{proof}

This proof motivates the name of a factorizing block system. Indeed, the maximal Schur set obtained in the end ``factors'' into the product of the sets $A$ and $B$, which are given in terms of the block system $\bB$.

In practice, \cref{thm:factor} is already useful in determining the uniform vertex-transitivity of graphs with relatively large automorphism group, graphs for which the computational methods in the last section alone are not sufficient. With a factorizing block system for a graph $\Gamma$, one can hope to find a maximal Schur set in $\Aut(\Gamma)$ by applying the methods from the previous section to the action of $\fix_\Gamma(\bB)$ and the action of $G/\!\fix_\Gamma(\bB)$ on $\bB$. This amounts to determining the clique number of two smaller graphs as opposed to the clique number of a larger graph. 

Computational evidence suggests that the converse of \cref{thm:factor} holds. Indeed, for every imprimitive graph which is known to be uniformly vertex-transitive (these will be discussed in the following section and in the appendix), there exists a factorizing block system for its automorphism group. Unfortunately, we are unable to prove the equivalence of uniform vertex-transitivity and the existence of factorizing block systems.

We finish this section with two more results on imprimitive graphs.

\begin{prop} \label{prop:simple}
If $\Gamma$ is an imprimitive graph such that $\Aut(\Gamma)$ is simple, then $\Aut(\Gamma)$ does not admit a factorizing block system.
\end{prop}
\begin{proof}
Let $\Gamma$ be an imprimitive graph and let $\bB$ be a block system for the action of $\Aut(\Gamma)$ on $V(\Gamma)$. As above, let each block of $\bB$ have size $k$. It is clear from the above construction that $\fix_\Gamma(\bB)$ is a normal subgroup of $\Aut(\Gamma)$, realized as the kernel of a homomorphism. Since $\Aut(\Gamma)$ is simple by assumption, we have that $\fix_\Gamma(\bB)$ is either trivial or all of $\Aut(\Gamma)$. If $\fix_\Gamma(\bB)$ is trivial, then it cannot contain a Schur set of size $k$, where $k>1$ necessarily from the definition of imprimitivity. Similarly, if $\fix_\Gamma(\bB)=\Aut(\Gamma)$, then $\Aut(\Gamma)/\fix_\Gamma(\bB)$ is trivial, and hence does not act uniformly transitively on $\bB$.
\end{proof}

\begin{prop} \label{prop:a5}
If $\Gamma$ is a vertex-transitive graph with $\Aut(\Gamma)\simeq A_5$, then $\Gamma$ is imprimitive.
\end{prop}
\begin{proof}
We show that $A_5$ does does not occur as the automorphism group of a primitive vertex-transitive graph. Recall from \cite[Corollary 1.5A]{DM96} that the primitive actions of a group $G$ are characterized by their point stabilizers being maximal subgroups of $G$. A transitive group action is determined up to conjugacy by its point stabilizer. Thus, the primitive actions of a group $G$ correspond to the maximal subgroups of $G$. In the case of $A_5$, there are three conjugacy classes of maximal subgroups, as seen in \cite{ATLAS}, for instance. These correspond to three primitive actions of $A_5$ on 5 points, 6 points, and 10 points, respectively. By consulting a database \cite{R97} of all possible vertex-transitive graphs on 5, 6, and 10 vertices, we verify that none of these primitive actions of $A_5$ occur as the automorphism group of a vertex-transitive graph.
\end{proof}

Note that in the case the converse of \cref{thm:factor} was true, \cref{prop:simple} and \cref{prop:a5} would imply that there is no uniformly vertex-transitive graph with automorphism group $A_5$.

\section{Experimental results}\label{sec:exp}

The basis for our experimental results is a database \cite{R97} of vertex-transitive non-Cayley graphs on up to 30 vertices\footnote{The website for this database remarks that it is ``guaranteed correct only up to 26 vertices.'' However, the completeness of this dataset is supported by later work by Royle \cite[Table 3]{HR19}, which lists the same counts of vertex-transitive non-Cayley graphs as appear in the dataset.}. Indeed, the motivation for much of our results was to develop methods to classify these graphs into the uniformly vertex-transitive and non-uniformly vertex-transitive ones. We carried out these methods using the SageMath environment \cite{sage} and GAP \cite{GAP4}. We first present a summary of our findings in \cref{tab:counts}. 

\begin{table}[ht] 
\begin{tabular}{|c|cc|ccc|} \hline
Vertices & VT & non-Cayley & UVT & non-UVT & Unknown \\ \hline\hline
10 & 22 & 2 & 2 & & \\
15 & 48 & 4 & & 4 & \\
16 & 286 & 8 & 8 & & \\
18 & 380 & 4 & 4 & & \\
20 & 1214 & 82 & 70 & 12 & \\
24 & 15506 & 112 & 112 & & \\
26 & 4236 & 132 & 132 & & \\
28 & 25850 & 66 & 24 & 38 & 4\\
30 & 46308 & 1124 & 324 & 730 & 70 \\\hline
\end{tabular}
\\[3mm]\caption{The counts of graphs with various properties on numbers of vertices for which there exist vertex-transitive non-Cayley graphs.}\label{tab:counts}\end{table}

Outside of two complementary pairs of graphs on 28 vertices, the classification of uniformly vertex-transitive graphs on less than vertices is complete. A complete descriptions of these graphs by their automorphism groups can be found in the appendix. One of the two remaining complementary pairs of graphs has automorphism group $2\times\PSL(2,13)$, of order 2184. This is, in practice, too large to compute using only the clique number of the orthogonality graph. However, it has been computed that none of the block systems for this automorphism group are factorizing, so the converse of \cref{thm:factor}, if true, would imply that this pair of graphs is not uniformly vertex-transitive.

The other complementary pair of graphs on 28 vertices consists of the Johnson graph $J(8,2)$ and its complement. This pair of graphs has automorphism group $S_8$ of order 40320 acting primitively, so determining the uniform vertex-transitivity of this graph is outside the means of our current methods. The smaller Johnson graphs $J(n,2)$ exhibit curious behavior. The graph $J(5,2)$ is the complement of the Petersen graph, which is uniformly-vertex transitive, but not Cayley. The graph $J(6,2)$ is the line graph of $K_6$, which is not uniformly vertex-transitive (indeed, it is the $S_6$ entry in the \cref{tab:nuvt}). Most curiously, the graph $J(7,2)$ is a Cayley graph for the group $7:3$. Further methods to determine the uniform-vertex transitivity of graphs which have a large primitive automorphism group would be desireable, in order to settle the case of $J(8,2)$ and other graphs with large primitive automorphism groups.

\begin{table}[ht]
\begin{tabular}{lcccc} \hline
Group $G$ & Degree & $\omega(D(G))-\deg(G)$ & Graphs & ID\# \\\hline
$2^3 : 7$ & 28 & $-21$ &18& $(28,11)$ \\
$A_5$ & 20 & $-10$ &4& $(20,15)$ \\
& 30 & $-17$ &382& $(30,9)$ \\
$2 \times A_5$ & 20 & $-8$ &8& $(20,36)$ \\
& 30 & $-4$ &88& $(30,29)$ \\
& 30 & $-4$ &32& $(30,30)$ \\
$S_5$ & 15 & $-2$ &2& $(15,10)$ \\
& 30 & $-17$ &22& $(30,22)$ \\
& 30 & $-4$ &32& $(30,25)$ \\
& 30 & $-14$ &90& $(30,27)$ \\
$\PSL(3,2)$ & 28 & $-18$ &4& $(28,32)$ \\
$2 \times S_5$ & 30 & $-4$ &40& $(30,58)$ \\
& 30 & $-4$ &44& $(30,60)$ \\
$\PSL(3,2) : 2$ & 28 & $-18$ &12& $(28,46)$ \\
$\PSL(2,8)$ & 28 & $-18$ &2& $(28,70)$ \\
$S_6$ & 15 & $-2$ &2& $(15,28)$ \\
$2^3 : \PSL(3,2)$ & 28 & $-16$ &2& $(28,159)$ \\
\hline \end{tabular}
\\[3mm]\caption{The vertex-transitive graphs which are not uniformly-vertex transitive, listed by automorphism group.}\label{tab:nuvt}\end{table}

We additionally list by automorphism group the vertex-transitive graphs which are known to not be uniformly-vertex transitive in \cref{tab:nuvt}. The first column lists a description of the permutation list and the second column lists the degree of the permutation group. The group descriptions are based on output from the GAP \texttt{StructureDescription()} function. Note that an abstract group may have several entries corresponding to its different actions. In the third column we record the quantity $\omega(D(G))-\deg(G)$, the difference between the maximum size of a clique in $D(G)$ and the degree of the permutation group $G$. If this value is 0, then $G$ is uniformly-transitive. The fourth column lists the number of graphs which have this automorphism group. Lastly, we record the identification number of $G$ in the GAP Transitive Groups Library \cite{TransGrp}. For example, to create the first group in the table in GAP, one would simply call \texttt{TransitiveGroup(28,11)}.

One noteworthy observation from \cref{tab:nuvt} is that their automorphism groups tend to include non-abelian simple groups as a large subgroup. In some sense, this may provide further evidence that the converse of \cref{thm:factor} is true. Indeed, the automorphism groups failing to have a factorizing block system because of a large simple subgroup is in line with \cref{prop:simple}.

We provide this information for all groups appearing as the automorphism group of a vertex-transitive non-Cayley graph (including the uniformly vertex-transitive graphs) in the appendix.

\section{Further directions}\label{sec:further}

One possible generalization of the notion of uniform-vertex transitivity is the following, in which we replace the matrix $J_n$ by one of its integer multiples.

\begin{defn}
Let $\Gamma$ be a graph on $n$ vertices and let $k\geq 1$. The graph $\Gamma$ is \emph{$k$-uniformly vertex-transitive} if there exists a size $kn$ subset $\{\sigma_1,\ldots,\sigma_{kn}\}\subset\Aut(\Gamma)$ such that, when viewed as permutation matrices, $\sum \sigma_i = kJ_n$. Such a subset $\{\sigma_1,\ldots,\sigma_{kn}\}$ is called a \emph{$k$-maximal Schur set}. 
\end{defn}

Setting $k=1$ recovers the above notion of uniform vertex-transitivity. It turns out that every vertex-transitive graph is $k$-uniformly vertex-transitive for some $k$, and this $k$ admits an explicit description in terms of the action of the group $\Aut(\Gamma)$. 

\begin{prop}\label{prop:uvtk}
Let $\Gamma$ be a vertex-transitive graph on $n$ vertices, and let $s$ be the order of the stabilizer of a vertex in $\Aut(\Gamma)$. Then,
\begin{enumerate}[(a)]
\item the graph $\Gamma$ is $s$-uniformly vertex-transitive.
\item for $1\leq k\leq s$, $\Gamma$ is $k$-uniformly vertex-transitive if and only if $\Gamma$ is $(s-k)$-uniformly vertex-transitive.
\end{enumerate}
\end{prop}
\begin{proof}
We claim that $\Aut(\Gamma)$ is an $s$-maximal Schur set for itself. Indeed, let $M=\sum_{\sigma\in\Aut(\Gamma)}\sigma$. Consider the $i$th row of the matrix $M$, which has an $s$ in the $i$th position, as $s$ is the number of elements in $\Aut(\Gamma)$ stabilizing vertex  $i$. The $j$th entry of the $i$th row corresponds to all automorphisms moving vertex $j$ to vertex $i$, which is a coset of the stabilizer of vertex $j$, which also has size $s$. Thus, each row of $M$ consists of $s$ in each position. This shows that $M=sJ_n$, which proves (a).

Suppose that $\Gamma$ is $k$-uniformly vertex-transitive for some $k\leq s$, and let $S$ be a $k$-maximal Schur set. Then, $\Aut(\Gamma)-S$ is a $(s-k)$-maximal Schur set. Indeed, we see that
\[ \sum_{\sigma\in \Aut(\Gamma) - S} \sigma = \sum_{\sigma \in \Aut(\Gamma)} \sigma - \sum_{\sigma \in S} \sigma = sJ_n - kJ_n = (s-k)J_n \] 
which proves (b).
\end{proof}

The following result improves \cref{prop:implications} for Cayley graphs.

\begin{prop}\label{prop:cay-uvtk}
Let $\Gamma$ be a Cayley graph on $n$ vertices, with vertex stabilizer of size $s$. Then $\Gamma$ is $k$-uniformly vertex-transitive for all $1\leq k\leq s$.
\end{prop}
\begin{proof}
Let $G$ be a subgroup of $\Aut(\Gamma)$ which acts regularly on $V(\Gamma)$, as guaranteed to exist by \cref{prop:cay-reg}. In particular, $|G|=n$, and $\Aut(\Gamma)$ is partitioned into $s$ many subsets of size $n$ by the cosets of $G$ in $\Aut(\Gamma)$. Each of these cosets is a maximal Schur set for $\Aut(\Gamma)$ by \cref{lem:clique_mult}. As such, taking the union of any $k$ of these cosets yields a $k$-maximal Schur set of $\Aut(\Gamma)$.
\end{proof}

One fundamental difference in $k$-uniform vertex-transitivity for $k=1$ versus $k>1$ is that for $k>1$, computation is more difficult. As seen, the $k=1$ case can be understood in terms of a pairwise property of automorphisms---that their entrywise product is 0. The $k>1$ does not seem to admit such a clean understanding, and therefore, more sophisticated methods are needed to determine whether or not an arbitrary graph is $k$-uniformly vertex-transitive for $k>1$.

In particular, it is unclear whether non-uniformly vertex-transitive graphs may be $k$-uniformly vertex-transitive for some $1<k<s$, with $s$ the order of a vertex stabilizer in the automorphism group. Even on the smallest example of a such a graph, the line graph of the Petersen graph, we were unsuccessful in determining whether it is 2-uniformly vertex-transitive. 

\section*{Acknowledgments} 
The authors would like to thank David Roberson for useful comments on an earlier version of this article and for bringing the notions of the derangement graph and sharply transitive sets to our attention. This work was supported by the DAAD RISE program and the collaborative research center SFB-TRR 195 \emph{Symbolic Tools in Mathematics and their Application}. Simon Schmidt and Moritz Weber were also supported by the DFG project \emph{Quantenautomorphismen von Graphen}.

\appendix\section{Automorphism groups of non-Cayley graphs}

Here we provide the information presented in \cref{tab:nuvt} for \emph{all} vertex-transitive non-Cayley graphs, including the uniformly vertex-transitive graphs. We leave the entry in the $\omega(D(G))-\deg(G)$ column blank if the value is unknown.
\begin{longtable}{lcccc} \hline\endfoot \hline Permutation Group $G$ & Degree & $\omega(D(G))-\deg(G)$ & Graphs & ID\# \\\hline\endhead

$2^3 : 4$  & 16 & $0$ &8& $(16,33)$ \\
$3^2 : 4$ & 18 & $0$ &2& $(18,10)$ \\
$2 \times (5 : 4)$ & 20 & $0$ &34& $(20,13)$ \\
$A_4 : 4$ & 24 & $0$ &20& $(24,51)$ \\
$4 \times A_4$ & 24 & $0$ &4& $(24,55)$ \\
${13} : 4$ & 26 & $0$ &130& $(26,4)$ \\
$2^3 : 7$ & 28 & $-21$ &18& $(28,11)$ \\
$A_5$ & 20 & $-10$ &4& $(20,15)$ \\
 & 30 & $-17$ &382& $(30,9)$ \\
${15} : 4$ & 30 & $0$ &44& $(30,6)$ \\
$3 \times (5 : 4)$ & 30 & $0$ &96& $(30,7)$ \\
$S_3^2 : 2$ & 24 & $0$ &64& $(24,72)$ \\
$(5 : 4) \times S_3$ & 30 & $0$ &88& $(30,24)$ \\
$2 \times A_5$ & 20 & $-8$ &8& $(20,36)$ \\
 & 30 & $-4$ &88& $(30,29)$ \\
 & 30 & $-4$ &32& $(30,30)$ \\
$S_5$ & 10 & $0$ &2& $(10,13)$ \\
 & 15 & $-2$ &2& $(15,10)$ \\
 & 30 & $-17$ &22& $(30,22)$ \\
 & 30 & $-4$ &32& $(30,25)$ \\
 & 30 & $-14$ &90& $(30,27)$ \\
$3^2 : QD_{16}$ & 18 & $0$ &2& $(18,73)$ \\
${13} : {12}$ & 26 & $0$ &2& $(26,8)$ \\
$2 \times (2^4 : 5)$ & 20 & $0$ &4& $(20,41)$ \\
$\PSL(3,2)$ & 28 & $-18$ &4& $(28,32)$ \\
$3^2 : (5 : 4)$ & 30 & $0$ &48& $(30,46)$ \\
$(((4 \times 2) : 4) : 3) : 2$ & 24 & $0$ &8& $(24,313)$ \\
$((2 \times (4 : 4)) : 2) : 3$ & 24 & $0$ &4& $(24,461)$ \\
$((2^2 \times D_8) : 2) : 3$ & 24 & $0$ &4& $(24,499)$ \\
$A_5 : 4$ & 24 & $0$ &4& $(24,578)$ \\
$2 \times S_5$ & 20 & $0$ &12& $(20,62)$ \\
 & 30 & $-4$ &40& $(30,58)$ \\
 & 30 & $-4$ &44& $(30,60)$ \\
$2 \times ((2^4 : 5) : 2)$ & 20 & $0$ &8& $(20,87)$ \\
$\PSL(3,2) : 2$ & 28 & $-18$ &12& $(28,46)$ \\
$((((4 \times 2) : 4) : 3) : 2) : 2$ & 24 & $0$ &4& $(24,844)$ \\
$5^2 : ((4 \times 2) : 2)$ & 20 & $0$ &2& $(20,96)$ \\
$\PSL(2,8)$ & 28 & $-18$ &2& $(28,70)$ \\
$2 \times ((2^4 : 5) : 4)$ & 20 & $0$ &4& $(20,140)$ \\
$(3^2 : (({10} \times 2) : 4)$ & 30 & $0$ &20& $(30,169)$ \\
$S_5 \times S_3$ & 30 & $0$ &16& $(30,174)$ \\
$S_6$ & 15 & $-2$ &2& $(15,28)$ \\
 & 30 &  &6& $(30,164)$ \\
$2 \times ((2^6) : 7)$ & 28 & $0$ &8& $(28,110)$ \\
$(2^3) : \PSL(3,2)$ & 28 & $-16$ &2& $(28,159)$ \\
$(A_6 : 2) : 2$ & 30 &  &8& $(30,264)$ \\
$2 \times S_6$ & 20 & $0$ &2& $(20,198)$ \\
 & 30 &  &12& $(30,260)$ \\
 & 30 &  &8& $(30,261)$ \\
$2 \times (((2^6) : 7) : 2)$ & 28 & $0$ &16& $(28,185)$ \\
$2 \times \PSL(2,13)$ & 28 &  &2& $(28,199)$ \\
$2 \times (2^4 : S_5)$ & 30 &  &8& $(30,517)$ \\
$5^3 : (A_4 : 4)$ & 30 & $0$ &4& $(30,604)$ \\
$5^3 : (4 \times S_4)$ & 30 &  &4& $(30,780)$ \\
$S_5 \wr 2$ & 20 & $0$ &4& $(20,540)$ \\
 & 30 &  &4& $(30,1011)$ \\
$S_8$ & 28 &  &2& $(28,502)$ \\
$3^5 : (2 \times ((2^4 : 5) : 4))$ & 30 & $0$ &8& $(30,1550)$ \\
$S_6 \wr 2$ & 30 &  &4& $(30,2029)$ \\
$2 \times (2^4 : (2^{10} : S_5))$ & 30 &  &4& $(30,2525)$ \\
$A_5^3 : (2 \times S_4)$ & 30 &  &4& $(30,2994)$ \\
$2 \times (2^{14} : S_6)$ & 30 &  &4& $(30,3397)$ \\
$3^{10} : (2^4 : ((2^5 : A_5) : 2^2))$ & 30 &  &4& $(30,5185)$ \\ \end{longtable}

\end{document}